\documentclass[12pt]{amsart}

\usepackage{amssymb,latexsym,amsmath,graphics,setspace,amscd,verbatim}
\usepackage{epsfig,epsf,amsthm,amsfonts}
\usepackage[active]{srcltx}

\theoremstyle{plain}
\newtheorem{theo}{Theorem}[section]
\newtheorem{prop}[theo]{Proposition}
\newtheorem{lemma}[theo]{Lemma}

\theoremstyle{definition}
\newtheorem{defi}[theo]{Definition}
\newtheorem{remark}[theo]{Remark}

\newcommand{\R}{\mathbb{R}}
\newcommand{\Z}{\mathbb{Z}}

\newcommand{\N}{\mathbb{N}}

\newcommand{\T}{\mathcal{T}}

\begin{document}

\baselineskip=17pt

\title{Persistence  of fixed points under rigid perturbations of maps}
\author{Salvador Addas-Zanata} \address[Salvador Addas-Zanata]{Departamento
de Matem\'atica Aplicada, Instituto de Matem\'atica e Estat\'istica - Universidade de
S\~ao Paulo} \email[Salvador Addas-Zanata]{sazanata@ime.usp.br}
\author{Pedro A. S. Salom\~ao} \address[Pedro A. S. Salom\~ao]{Departamento
de Matem\'atica, Instituto de Matem\'a\-tica e Estat\'istica - Universidade de
S\~ao Paulo} \email[Pedro A. S. Salom\~ao]{psalomao@ime.usp.br}
\subjclass[2000]{Primary 37C25, 37C50, 37E30}
\keywords{Topological Dynamics, Brouwer theory, Generating functions}

\begin{abstract}
Let $f:S^1\times [0,1]\to S^1\times [0,1]$ be a real-analytic annulus diffeomorphism which is homotopic to the identity map and preserves an area form. Assume that for some lift $\tilde {f}:\mathbb{R}\times [0,1]\rightarrow \mathbb{R}\times [0,1]$
we have ${\rm Fix}(\tilde{f})=\mathbb{R}\times \{0\}$ and that $\tilde{f}$ positively translates points in $\mathbb{R}\times \{1\}$. Let
$\tilde{f}_\epsilon $ be the perturbation of $\tilde{f}$ by the
rigid horizontal translation $(x,y)\mapsto (x+\epsilon ,y)$. We show that for all $\epsilon >0$ sufficiently small we have ${\rm Fix}
(\tilde{f}_\epsilon )=\emptyset $. The proof follows from
Ker\'ekj\'ar\-t\'o's construction of Brouwer lines for orientation
preserving homeomorphisms of the plane with no fixed points. This result
turns out to be sharp with respect to the regularity assumption: there
exists a diffeomorphism $f$ satisfying all the properties above, except that
$f$ is not real-analytic but only smooth, so that the above conclusion is
false. Such a map is constructed via generating functions.
\end{abstract}
\maketitle

\section{Introduction}

Let us denote by ${\rm Diff}^k(\mathbb{D})$ the set of orientation and area
preserving $C^{k\geq 1}$-diffeomorphisms $\hat h:\mathbb{D}\to \mathbb{D}$,
defined in the closed disk $\mathbb{D}:=\{z \in \mathbb{R}^2:|z| \leq 1\}$,
which fixes the origin $0 \in \mathbb{D}$. We denote by ${\rm Diff}_0^k(%
\mathbb{D})\subset {\rm Diff}^k(\mathbb{D})$ the subset of diffeomorphisms
satisfying
$$
{\rm Fix}(\hat h):=\{\hat h(z)=z\}=\{0\} \mbox { and } D \hat h(0)=Id.
$$
Here we are considering the usual area form $dz_1 \wedge dz_2$ on $\mathbb{R}^2$
with coordinates $(z_1,z_2)$.

In this paper we address the following question:

\vspace{0.3cm}

{\bf (Q1)} under what conditions can we find $\hat g\in {\rm Diff}^k(\mathbb{D})$
arbitrarily $C^k$-close to $\hat h$ so that ${\rm Fix}(\hat g)=\{0\}$ and $D \hat g(0)=e^{2
\pi \epsilon i}$, $\epsilon \in \mathbb{R}\setminus \mathbb{Q}$?

\vspace{0.3cm}



Before stating the main results we need some definitions.

\begin{defi}
\begin{enumerate}
\item  Let $A:=S^1\times [0,1]$ be the closed annulus, where $S^1$ is
identified with $\mathbb{R}/\mathbb{Z}$. Let $\tilde A:=\mathbb{R}\times
[0,1]$ be the infinite strip and $p:\tilde A\to A$ be the covering map $%
(x,y)\mapsto (x\mbox{ mod }1,y)$.

\item  Let $p_1:\tilde A\to \mathbb{R}$ and $p_2:\tilde A\to \R$ be the
projections of $\tilde A$ into the first and second factors, respectively.
We also denote by $p_1$ and $p_2$ the respective projections defined on $A$.

\item  Let ${\rm Diff}^k(A)$ be the space of area preserving $C^k$
-diffeo\-morphisms $f:A\to A$, where $k\in \mathbb{N}\cup \{\infty ,\omega \}
$, which are homotopic to the identity map. Let ${\rm Diff}_0^k(A)\subset
{\rm Diff}^k(A)$ denote the diffeomorphisms which satisfy the following
conditions: $f(x,0)=(x,0),\forall x\in S^1$ and if $\tilde f:\tilde A\to
\tilde A$ is the lift of $f$ such that $\tilde f(x,0)=(x,0),\forall x\in
\mathbb{R}$, then ${\rm Fix}(\tilde f)=\mathbb{R}\times \{0\}.$ Moreover, we require that
\begin{equation}
\label{twist} p_1 \circ \tilde f(x,1)>x,\forall x \in \mathbb{R}.
\end{equation}

\item  Let ${\rm Diff}_0^k(\tilde A)$ be the lifts of maps in ${\rm Diff}%
_0^k(A)$ which fix all points in $\mathbb{R}\times \{0\}$.
\end{enumerate}
\end{defi}

Now if $\hat h\in {\rm Diff}_0^k(\mathbb{D})$, we obtain a map $f:=b^{-1} \circ \hat h
\circ b$ induced by $b:A \to \mathbb{D}$,  defined
by
$$
b(x,y):=(\sqrt{y} \cos 2 \pi x,-\sqrt{y} \sin 2 \pi x),
$$
where $(x,y)$ are coordinates in $A$. Notice that $f$ preserves the area
form $dx \wedge dy$. We assume that $f$ extends to a map in  ${\rm Diff}^k(A).$ Clearly, $S^1\times \{0\}$ corresponds to the
blow up of $0\in \mathbb{D}$ and $S^1\times \{1\}$ corresponds to $\partial
\mathbb{D}$. Also, since $\hat h\in {\rm Diff}_0^k(\mathbb{D})$, it follows that either $f$ or $f^{-1}$
admits a lift $\tilde f \in {\rm Diff}_0^k(\tilde A)$. In fact, either $p_1
\circ \tilde f(x,1)>x,\forall x \in \mathbb{R}$ or $p_1 \circ \tilde
f(x,1)<x,\forall x \in \mathbb{R}.$ After possibly interchanging $f$ with $%
f^{-1}$ we may assume without loss of generality that \eqref{twist} is satisfied.

Given $\epsilon \in \mathbb{R}$ we consider the diffeomorphism
\begin{equation}
\label{translation} \tilde f_\epsilon:\tilde A \to \tilde A: (x,y) \mapsto
\tilde f(x,y)+(\epsilon,0).
\end{equation}
The map $\tilde f_\epsilon$ naturally induces a diffeomorphism $f_\epsilon:A
\to A$ given by
\begin{equation}
\label{translation2} f_\epsilon = p \circ \tilde f_\epsilon \circ p^{-1}.
\end{equation}
Notice that the translated map $f_\epsilon$ corresponds to blowing up the
map $\hat h \in {\rm Diff}_0^k(\mathbb{D})$ after compounding it with the rigid
rotation $z \mapsto e^{2\pi \epsilon i} z$.

Our first result is the following theorem.

\begin{theo}
\label{main1} Let $f\in {\rm Diff}_0^\omega (A)$ and $\tilde f\in {\rm Diff}%
_0^\omega (\tilde A)$ be a lift of $f$. Then there
exists $\epsilon _0>0$ such that for all $0<\epsilon <\epsilon _0$, we have  $
{\rm Fix}(\tilde f_\epsilon)=\emptyset$.
\end{theo}

\begin{remark}
The hypothesis $\tilde f(x,0)=(x,0),\forall x\in \mathbb{R}$ can be weakened
to $p_1\circ \tilde f(x,0)\geq x,\forall x\in \mathbb{R},$ as is easily seen
from the proof.
\end{remark}

\begin{remark}
From the classical Poincar\'e-Birkhoff theorem $\tilde f_\epsilon $ has
fixed points in ${\rm interior}(\tilde A)$ for all $\epsilon <0$
sufficiently small.
\end{remark}

Our next result proves sharpness of the real-analyticity assumption in
Theorem \ref{main1}, i.e, this phenomenon does not occur assuming only
smoothness.

\begin{theo}
\label{main2} There exist $f\in {\rm Diff}_0^\infty (A)$ and
 a sequence of positive real numbers $\epsilon _n\to 0^{+}$ as $
n\to \infty $ such that ${\rm Fix} (\tilde f_{\epsilon _n})\neq \emptyset$, where $
\tilde f\in {\rm Diff}_0^\infty (\tilde A)$ is the special lift of $f$ and $\tilde f_{\epsilon_n}$ is defined as in \eqref{translation}, for all $n\in \N$.
\end{theo}

The proof of Theorem \ref{main1} strongly relies on a construction due to B.
de Ker\'ekj\'ar\-t\'o \cite{kerek1} of Brouwer lines for orientation
preserving homeomorphisms of the plane which have no fixed point. Here, the
hypothesis of real-analyticity of $f$ plays an important role. We argue indirectly assuming the existence of a sequence $\epsilon_n \to 0^+$ such that $\tilde f_{\epsilon_n}$ admits a fixed point $z_n$. We can assume that $z_n$ converges to a point $\bar z$ at the lower boundary component of $\tilde A$. The real-analyticity hypothesis then allows one to conclude the existence of a small real analytic curve $\gamma_0$ starting at $\bar z$, which is a graph in the vertical direction, so that $\tilde f$ moves its point horizontally to the left. Since $\tilde f$ has no fixed point in ${\rm interior}(\tilde A)$, the curve $\gamma_0$ is then prolonged to a Brouwer line $L \subset \tilde A$, following Ker\'ekj\'ar\-t\'o's construction. We analyse all possibilities for the behaviour of $L$ and each of them yields a contradiction. Here, we strongly use the fact that $\tilde f$ moves points in the upper boundary of $\tilde A$ to the right.

The smooth map $f$ in Theorem \ref{main2} is obtained from a special generating function on $\tilde A$. More precisely, first we define a diffeomorphism $\psi:\tilde A \to \tilde A$ supported in the sequence of balls $B_k\subset \tilde A$ centered at $(0,3/2^{k+2})$ and radius $1/2^{k+3}$, converging to the origin. Using the function $h(t)=e^{-1/t}$, which extends smoothly at $t=0$ as a flat point, we define the generating function by $g(p)=h \circ p_2 \circ  \psi(p)$, where $p_2$ is the projection in the vertical direction. The diffeomorphism associated to $g$, which is a priori defined only in a small neighbourhood of the origin, is then suitably re-scaled in order to find the diffeomorphism $f$ of the annulus satisfying all the requirements.

As one can see, $f$ satisfies all hypotheses of Theorem \ref{main1} except that it is not real-analytic at a unique point in the lower boundary. This follows from the flatness of $h$ at $t=0$ and therefore the example in Theorem \ref{main2} shows the sharpness of the regularity assumption in Theorem \ref{main1}.

\section{Ker\'ekj\'art\'o's construction of Brouwer lines}

\label{SecKerek}

In this section we denote by $h:\mathbb{R}^2 \to \mathbb{R}^2$ an
orientation preserving homeomorphism of the plane satisfying
\begin{equation}
\label{fixa} {\rm Fix}(h)= \emptyset.
\end{equation}
The following periodicity in $x$ is assumed
\begin{equation}
\label{perix} h(x+1,y)=h(x,y)+(1,0),\forall(x,y)\in \mathbb{R}^2.
\end{equation}

\begin{defi}
\begin{enumerate}
\item[(a)]  We call $\alpha \subset \mathbb{R}^2$ a simple arc if $\alpha $
is the image of a topological embedding $\psi :[0,1]\to \mathbb{R}^2$. We
may consider the parametrization $\psi $ of the arc $\alpha $, which will
also be denoted by $\alpha $. We also identify all the parameterizations of $%
\alpha $ which are induced by orientation preserving homeomorphisms of the
respective domains. The internal points of the simple arc $\alpha $ are
defined by $\alpha \setminus \{\alpha (0),\alpha (1)\}$ and denoted $\dot
\alpha $. Given distinct points $B_1,B_2,\ldots ,$ in $\R^2$, we denote by $B1B2\ldots $
the polygonal arc connecting them by straight segments of lines following
that order. We may also denote by $AB$ a simple arc with endpoints $A \neq B$,
which is not necessarily a line segment.

\item[(b)]  Given any two simple arcs $\eta _0$ and $\eta _1$ with a unique
common end point, we denote by $\eta _0\cup \eta _1$ the simple arc obtained
by concatenating $\eta _0$ and $\eta _1$ in the usual way and respecting the
orientation from $\eta _0$ to $\eta _1$.

\item[(c)]  We say that the simple arc $\alpha \subset \mathbb{R}^2$ is a
translation arc if $\alpha (0)=z$, $\alpha (1)=h(z)\neq z$ and
$$
\alpha \cap h(\alpha )=\{h(z)\}.
$$

\item[(d)]  Let $\alpha $ be a simple arc with end points $b$ and $c$. We
say that $\alpha $ abuts on its inverse or direct image, respectively, if $b\not \in h^{-1}(\alpha )\cup
h(\alpha )=\emptyset $ and one of the following conditions holds:

\begin{enumerate}
\item[(i)]  $\dot \alpha \cap h^{-1}(\alpha )=\emptyset $ and $c\in h^{-1}(\alpha
)$.

\item[(ii)]  $\dot \alpha \cap h(\alpha )=\emptyset $ and $c\in h(\alpha )$.

\end{enumerate}

\item[(e)]  We say that $L\subset \mathbb{R}^2$ is a Brouwer line for $h$ if
$L$ is the image of a proper topological embedding $\psi :\mathbb{R}\to
\mathbb{R}^2$ so that $h(L)$ and $h^{-1}(L)$ lie in different components of $
\mathbb{R}^2\setminus L$.
\end{enumerate}
\end{defi}



Let $AB$ be a translation arc with end points $A$ and $B:=h(A)$. Let $C=h(B)$
and denote by $BC$ the simple arc given by $h(AB)$. Let us assume without
loss of generality that the vertical line passing through $B$ intersects the
arcs $AB$ and $BC$ only at $B$. Otherwise, we can perform a topological change of coordinates in order to achieve this property.

We will construct two half lines $L_1$ and $L_2$ issuing from $B$, with $L_1$
starting upwards and $L_2$ starting downwards,  so that $L=L_1 \cup L_2$
is a Brouwer line for $h$. $L_1$ and $L_2$ will be referred to as half Brouwer lines since both are topological embeddings of $[0,\infty)$ into $\R^2$ and $h(L_i) \cap L_i = \emptyset, i=1,2.$

Let us start with $L_1$. Consider the vertical
arc $\gamma_1$ starting upwards from $B$ which is defined by $%
\gamma_1(t)=B+(0,t),$ where $t\in[0, t^*]$ ($t^*$ to be defined below), or $t\in
[0,\infty)$. One of the following conditions is met:

\begin{enumerate}
\item[(i)]  There exists $t^{*}>0$ such that $\gamma _1$ abuts on its
inverse image and $P:=\gamma _1(t^{*})$ is such that $h(P)=:P^{\prime }$ is
an internal point of $\gamma _1$.

\item[(ii)]  There exists $t^{*}>0$ such that $\gamma _1$ abuts on its image
and $P:=\gamma _1(t^{*})$ is an internal point of $h(\gamma _1)$. In this
case we set $P^{\prime }:=h^{-1}(P)$ which is an internal point of $\gamma _1
$.

\item[(iii)]  $\gamma _1$ is defined for all $t\geq 0$, $(h^{-1}(\dot \gamma
_1)\cup \dot \gamma _1\cup h(\dot \gamma _1))\cap (AB\cup BC)=\emptyset $
and $h(\gamma _1)\cap \gamma _1=\emptyset .$
\end{enumerate}

In case (iii) our construction of $L_1$ ends and we define $L_1=\gamma_1$.
Otherwise in cases (i) and (ii), we define $PP^{\prime}$ to be the simple
arc in $\gamma_1$ from $P^{\prime}$ to $P$. Notice that by construction $%
PP^{\prime}$ is a translation arc. Ker\'ekj\'art\'o proves the following
theorem.

\begin{theo}[See \cite{kerek1}, Theorems II, III and IV]
\label{teofreeside} In cases (i) and (ii) above, we have
$$
h(\gamma _1)\cap AB=h^{-1}(\gamma _1)\cap BC=h(\gamma _1)\cap h^{-1}(\gamma
_1)=\emptyset .
$$

Moreover, in case (i) there exists a sub-arc $\nu _1$ of $h^{-1}(\gamma _1)$
from $A$ to $P$ such that $\nu _1\cup PP^{\prime }\cup h(\gamma _1)\cup
BC\cup AB$ is a simple closed curve which bounds an open domain $U_1\subset
\mathbb{R}^2$. In case (ii) there exists a sub-arc $\nu _1$ of $h(\gamma _1)$
from $C$ to $P$ such that $\nu _1\cup PP^{\prime }\cup h^{-1}(\gamma _1)\cup
AB\cup BC$ is a simple closed curve which bounds an open domain $U_1\subset
\mathbb{R}^2$.
\end{theo}

\begin{defi}
The free side of $PP^{\prime }$ is defined to be the side of $PP^{\prime }$
towards outside $U_1$ as in Theorem \ref{teofreeside}. See Figure $1$.
\end{defi}

The free side of the translation arc $PP^{\prime}\subset \gamma_1$ only
depends on which side $\gamma_1$ lies with respect to the oriented arc $AB
\cup BC$ and on how $\gamma_1$ abuts its image according to cases (i) or
(ii). This dependence strongly follows from the assumption that $h$ has no
fixed points and is exemplified in Figure $1$.

\begin{figure}\label{figu1}
\begin{center}
\includegraphics[width=80mm]{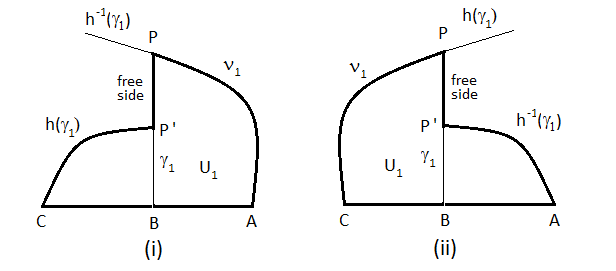}
\end{center}
\caption{In this picture, $\gamma_1$ abuts on its inverse and direct image as in cases (i) and (ii), respectively.}
\end{figure}

Now we need a couple of definitions in order to start the construction of
$L_1$.

\begin{defi}
\begin{enumerate}
\item  Let $R_n=[-n,n]\times [-n,n],\forall n\in \mathbb{N}^{*}$ and
$$
\epsilon _n=\inf \{|h(x)-x|,|h^{-1}(x)-x|:x\in R_n\}>0.
$$
Define $\eta _n>0$ to be the largest number $t\in (0,\epsilon _n/2]$ so that
$|h(x)-h(y)|\leq \epsilon _n/2$ and $|h^{-1}(x)-h^{-1}(y)|\leq \epsilon _n/2,
$ whenever $x,y\in R_n$ and $|x-y|\leq t$.

\item  Let $n\in \mathbb{N}^{*}$ and assume $PP^{\prime }\subset R_n$. By
mid-segment of $PP^{\prime }$ we mean a segment $M\subset PP^{\prime }$
so that the distances of its points to $P$ and to $P^{\prime }$ are at least
$\eta _n$. Notice that $M\neq \emptyset $.

\item  A base-point associated to the vertical translation arc $PP^{\prime }$ and to
a given free side of $PP^{\prime }$ is a point $B_1$ in a mid-segment $%
M\subset PP^{\prime }$ such that either the half line $l_{B_1}$ starting
from $B$ towards the free side of $PP^{\prime }$ is such that $l_{B_1}\cap
(h(l_{B_1})\cup h(PP^{\prime })\cup h^{-1}(PP^{\prime }))=\emptyset $ or
there exists a simple arc $\beta $ starting from $B_1$, perpendicular to $%
PP^{\prime }$ and towards the free side of $PP^{\prime }$ such that $\beta $
abuts on its image and $\beta \cap (h(PP^{\prime })\cup h^{-1}(PP^{\prime
}))=\emptyset $. In the former case, we say that the base point $B_1$ with
that given free side is unbounded and in the latter case we say that the
base point $B_1$ with that given free side is bounded. One of the endpoints
of $\beta $ is $B_1$ and the other is denoted by $P_1$.
\end{enumerate}
\end{defi}

The proof of the existence of at least one base point associated to a
translation arc $PP^{\prime}$ and to any given free side of $PP^{\prime}$ is
found in \cite[Section 2.2]{kerek1}.

\begin{remark}
If the translation arc $PP^{\prime }$ is horizontal, then the definitions
above are the same and analogous results hold.
\end{remark}

Continuing our construction, we find a base point $B_1$ associated to the
vertical translation arc $PP^{\prime}$. The initial part of $L_1$ is then
defined to be the segment $BB_1$. If $B_1$ is unbounded then we are finished
and $L_1= BB_1 \cup l_{B_1}$ is the desired half line. If $B_1$ is bounded
then the horizontal segment $\beta=B_1P_1$ abuts on its image and we find an
internal point $P_1^{\prime}=h(P_1)$ or $P_1^{\prime}=h^{-1}(P_1)$ as before such that
the horizontal arc $P_1P_1^{\prime}\subset \beta$ is a translation arc. The
translation arc $P_1P_1^{\prime}$ admits a free side according to the
description above. Observe that now the free side of $P_1P_1^{\prime}$ is
either the upper or the lower side. Again we find a base point $B_2\subset
P_1P_1^{\prime}$ towards the free side of $P_1P_1^{\prime}$ and add the
simple arc $B_1B_2$ to $L_1$, now given by $L_1=BB_1B_2$. Repeating this
procedure indefinitely we arrive at one of the following cases:

\begin{enumerate}
\item[(i)]  after a finite number of steps we find an unbounded base point $
B_j\in P_{j-1}P_{j-1}^{\prime }$ and our broken half line is given by $
L_1=BB_1B_2\ldots B_jl_{B_j}$.

\item[(ii)]  all base points $B_j$ found in the construction are bounded and we define $
L_1=BB_1B_2B_3\ldots $. Then the following holds: given $n\in N^{*},$ there
exists $k_0\in N^{*}$ such that $B_k\not \in R_n,\forall k\geq k_0$. This
follows from the definition of base points and is proved in \cite{kerek1}.
\end{enumerate}

Notice that the construction of $L$ depends on the choices of the internal
base points $B_k \in P_{k-1}P_{k-1}^{\prime}$. Also, the half line $L_1$
goes to infinity and
\begin{equation}
\label{eqL1} h(L_1) \cap L_1 = (h(\dot L_1)\cup h^{-1}(\dot L_1) \cup \dot
L_1) \cap (AB \cup BC) = \emptyset.
\end{equation}

We still need a modification trick from \cite{kerek1} in the construction of
$L_1$. It is called the deviation of the path. Let $V_k=\{(x,y)\in
\mathbb{R}^2:x=k\},k\in \mathbb{Z},$ be the vertical lines at integer values and assume that

\begin{equation}
\label{finitude}0<l:=\#V_0\cap h^{-1}(V_0)<\infty.
\end{equation}

Notice that from \eqref{perix}, hypothesis \eqref{finitude} must hold as
well for each $V_k,k\in \mathbb{Z},$ and the respective intersections are
shifted by $(k,0)$.

Let $V_0\cap h^{-1}(V_0)=\{w_1,\ldots ,w_l\}$ and $w_i^{\prime
}:=h(w_i),i=1,\ldots ,l$. Consider the vertical arcs $\gamma
_i=w_iw_i^{\prime }\subset V_0,i=1,\ldots ,l$. If $\gamma_j$ does not
properly contains any $\gamma_i$ with $i\neq j$, then $\gamma_j$ is a
translation arc. We consider only such translation arcs on $V_0$ and keep
denoting them by $\gamma_j$, now with $j=1,\ldots ,l_0,l_0\leq l$. Given $j$%
, assume a free side of $\gamma_j$ is given and is to the left. Then there
exists a base point $u_{j,l}\in \dot \gamma _j$ associated to $\gamma_j$
and to that free side. Accordingly, if the given free side of $\gamma_j$ is
to the right we can also find a base point $u_{j,r}\in \dot \gamma _j$
associated to $\gamma_j$ and to that free side. Let $\gamma _j^i=\gamma
_j+(i,0)$ be the respective translation arcs on $V_i$ for all $i\in
\mathbb{Z}$ and let $u_{j,l}^i:=u_{j,l}+(i,0),u_{j,r}^i:=u_{j,r}+(i,0),i\in
\mathbb{Z}$ be their respective base points. In the following we fix these
base points $u_{j,l}^i$ and $u_{j,r}^i$ in each $\gamma _j^i$.

In the construction of $L_1$ above suppose that at some point we find a
vertical translation arc $P_{k-1}P^{\prime}_{k-1}$ with a given free side
and the horizontal arc issuing from a bounded base point $B_k\subset
P_{k-1}P^{\prime}_{k-1}$ towards the free side intersects some $V_j$ at an
internal point $z \in B_kB_{k+1}$ so that the arc $B_kz$ intersects no other
vertical $V_i,i\neq j$. Instead of adding the segment $B_kB_{k+1}$ to $L_1$
we add only the segment $B_kz$ and the new $B_{k+1}$ is determined according
to one of the alternatives found in the following theorem.

\begin{theo}[\cite{kerek1},\cite{gui}]
We have

\begin{enumerate}
\item[(i)]  there exists a vertical half line $l_z\subset V_j$ through $z$
so that the broken half-line $\alpha :=B_kz\cup l_z$ satisfies $\alpha \cap
h(\alpha )=\emptyset $ and $(h(P_{k-1}P_{k-1}^{\prime })\cup
h^{-1}(P_{k-1}P_{k-1}^{\prime }))\cap \alpha =\emptyset $. In this case we
have $L_1=BB_1\ldots B_kzl_z$ and the construction of $L_1$ is finished.

\item[(ii)]  there exists $c\in V_j, c\neq z,$ so that the broken arc $\alpha
:=B_kz\cup zc$ abuts on its image, satisfies $(h(P_{k-1}P_{k-1}^{\prime
})\cup h^{-1}(P_{k-1}P_{k-1}^{\prime }))\cap \alpha =\emptyset ,$ and
contains a translation arc $\gamma _m^j\subset V_j$ for some $m\in
\{1,\ldots ,l_0\}$. Let $B_{k+1}\in \{u_{m,l}^j,u_{m,r}^j\}$ be the base point
in $\gamma_m^j$ associated to the free side of $\alpha $. In this case we
have $L_1=BB_1\ldots B_kzB_{k+1}\ldots $ and we keep constructing $L_1$
through the horizontal arc issuing from $B_{k+1}\in V_j$ towards the free
side of $\alpha $ as before.

\item[(iii)]  the point $z$ is an internal point of a translation arc $
\gamma_m^j$ for some $m$ so that $\gamma _m^j\cap (h(P_{k-1}P_{k-1}^{\prime
})\cup P_{k-1}P_{k-1}^{\prime }\cup h^{-1}(P_{k-1}P_{k-1}^{\prime
}))=\emptyset $ and $\cup _{n\in \mathbb{Z}}h^n(\gamma_m^j)\cap B_kz=\{z\}.$
In this case the free side of $\gamma _m^j$ is the side
opposite to $B_kz$ and we find a base point $B_{k+1}\in
\{u_{m,l}^j,u_{m,r}^j\}$ associated to $\gamma _m^j$ and to that free side.
We have $L_1=BB_1\ldots B_kzB_{k+1}\ldots $ and we keep constructing $L_1$
through the horizontal arc issuing from $B_{k+1}$ towards that free side.
\end{enumerate}
\end{theo}

\begin{figure}\label{figu4}
\begin{center}
\includegraphics[width=100mm]{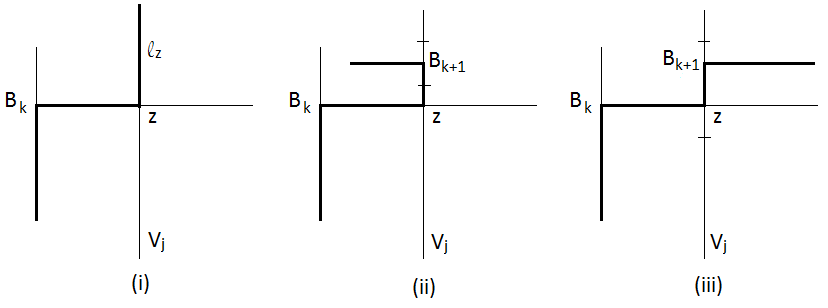}
\end{center}
\caption{The deviation of the path according to cases (i), (ii) and (iii) above. }
\end{figure}

Proceeding as above indefinitely and using the deviation of the path
whenever its conditions are met we find the desired half line $L_1$. As
explained before, $L_1$ goes to infinity and satisfies \eqref{eqL1}.

\begin{remark}
\label{remboun} Given $a\in \mathbb{R}$, let $G_a=\{(x,y)\in \mathbb{R}%
^2:y\geq a\}.$ We can extract some more information one how $L_1$ goes to
infinity if we assume the following twist condition on $h$
\begin{equation}
\label{twistinfinito}\exists y_0\in \mathbb{R}\mbox{ s.t. }p_1\circ h(x,y)>x%
\mbox{  and  }p_1\circ h^{-1}(x,y)<x,\forall (x,y)\in G_{y_0}.
\end{equation}
Suppose that at some point in the construction of $L_1$ we find a
horizontal translation arc $P_{k-1}P_{k-1}^{\prime }$ with the free side
coinciding to its upside and an associated base point $B_k\in
P_{k-1}P_{k-1}^{\prime }$. Assume that there exists a vertical segment $V$
starting from $B_k$ towards the free side so that its other extremity lies
inside $G_{y_0}$ and that $h(V)\cap V=\emptyset $. We claim that $B_k$ is an
unbounded base point and the construction of $L_1$ ends by adding to it the
vertical half line $l_{B_k}$,i.e., $L_1=BB_1\ldots B_kl_{B_k}$. To see this
we argue indirectly and assume the existence of a vertical segment $W$
starting from $B_k$ and containing $V$ such that $W$ abuts on its image. Let
$w\neq B_k$ be the other extremity of $W$. Then either $h(w)\in W$ or $%
h^{-1}(w)\in W$. However, this contradicts \eqref{twistinfinito} and proves
our claim.
\end{remark}

\begin{remark}
\label{remdev} Using the deviation of the path explained above we know that
if $L_1$ is horizontally unbounded then $L_1$ is eventually periodic. This
follows from the finiteness of the points $u_{j,l}^i,u_{j,r}^i\in V_i$ for
each $i\in \mathbb{Z}$. For instance, suppose that  $L_1$
is deviated at some $V_i,$ for some $i\in \Z$, and leaves it to the right at $u_{j,r}^i\in V_i$, for some $j\in \{1,...,l_0\}$. Suppose that after this deviation, $L_1$ is  now deviated at $V_{i+N},$ for some $N\in \mathbb{Z}^*,$  leaving it to the right at $u_{j,r}^{i+N}\in V_{i+N}.$ We continue the construction of $L_1$ from $u_{j,r}^{i+N},$  proceeding in exactly the same way as we did from $u_{j,r}^i.$ This implies that, except perhaps for its initial segments, $L_1$
is periodic, i.e., there exists a connected subset $W_0 \subset L_1$
from $u_{j,r}^i$ to $u_{j,r}^{i+N}$ so that $W_0 +(kN,0) \subset L,\forall k\in \N$. Hence we find a new periodic Brouwer line $L_{{\rm per}}$ given by $W_0+(kN,0),k\in \mathbb{Z}.$ By
construction we must have $h(L_{{\rm per}})\cap L_{{\rm per}}=\emptyset $.
\end{remark}
%
%
%
%
\begin{remark}
\label{remL2} The construction of the other half line $L_2$ from $B$ (now
starting downwards) can be done in exactly the same way as we did for $L_1$
so that by construction $L=L_1\cup L_2$ is a Brouwer line. An alternative
construction for $L_2$, which will be used in the proof of Theorem \ref
{main1}, is the following: let $\psi :[0,\infty )\to \mathbb{R}^2$ be a
proper topological embedding with $\psi (0)=B$ and let $L_2=\psi ([0,\infty
))$. Assume the $h(L_2)\cap L_2=\emptyset $ and that $\dot L_1$ and $\dot L_2
$ lie in different components of $\mathbb{R}^2\setminus (h^{-1}(L_2)\cup
AB\cup BC\cup h(L_2))$. Then one easily checks that $L=L_1\cup L_2$ is a
Brouwer line for $h$.
\end{remark}

We end this section with a proposition that will be useful in the proof of Theorem~\ref{main1} in the next section. Its proof is entirely contained in Ker\'ekj\'ar\-t\'o's construction of Brouwer lines explained above.

\begin{prop}\label{propresume}Let $h:\R^2 \to \R^2$ be an orientation preserving homeomorphism of the plane which has no fixed points and satisfies the following assumptions:
\begin{itemize}
\item[(i)]  There exists $y_0\in \mathbb{R}$ such that $$p_1\circ h(x,y)>x \mbox{ and } p_1\circ h^{-1}(x,y)<x,\forall (x,y)\in \R^2,y\geq y_0.$$
\item[(ii)] $h(x+1,y) = h(x,y)+(1,0),\forall (x,y)\in \R^2.$
\item[(iii)] There exists a vertical line $V_0$ such that $0<l:=\#V_0\cap h^{-1}(V_0)<\infty.$
 \end{itemize} Then through any point $B\in \R^2$ as above, there exists a half Brouwer line $L_1$ issuing from $B$ upwards so that the following holds:
\begin{itemize}
\item $L_1$ contains only horizontal and vertical segments.
\item if $L_1$ contains a point $q\in\{y \geq y_0\}$ then it contains the vertical upper half line through $q$.
\item if $L_1$ is horizontally unbounded then $L_1$ is eventually periodic, i.e., there exists a simple arc $W_0\subset L_1$ and an integer $N\neq 0$ so that $W_0 + (kN,0)\subset L_1,\forall k\in \N$. $|N|$ is the least positive integer with this property. In particular, this implies that $W_0 \cap W_0 + (N,0)=\{{\rm point}\}.$
\item if $L_2$ is a given half Brouwer line issuing from $B$ downwards and $\dot L_1$ and $\dot L_2$ lie in different components of $\mathbb{R}^2\setminus (h^{-1}(L_2)\cup AB\cup BC\cup h(L_2))$, then $L_1 \cup L_2$ is a Brouwer line.
\end{itemize}Here, as above, $B = h(A), C=h(B)$, $AB$ is a translation arc and $BC=h(AB)$ is horizontal.
 \end{prop}

\section{Proof of Theorem \ref{main1}}

We start with the following lemma.

\begin{lemma}
\label{lemfix}Let $\tilde g:\tilde U\subset \tilde A\to \tilde A$ be a real
analytic area-preserving diffeomorphism defined in an open neighbourhood $%
\tilde U$ of $\mathbb{R}\times \{0\}\subset \tilde A$ so that ${\rm Fix}(\tilde
g)=\mathbb{R}\times \{0\}$. Assume that there exists a sequence of positive
real numbers $\epsilon _n\to 0^{+}$ such that each $\tilde g_{\epsilon _n}$,
defined as in \eqref{translation}, admits a fixed point $p_n\in \tilde U,\forall n,$
with $p_n\to \bar p=(\bar x,0)\in \mathbb{R}\times \{0\}$ as $n\to \infty $.
Then there exists a real-analytic curve $\gamma _0:[0,1]\to \tilde U$, $\gamma
_0(t)=(x(t),y(t))$ so that $\tilde g\circ \gamma _0(t)=(w(t),y(t))$ and  it satisfies
\begin{equation}
\label{lemres}w(0)=\bar x,w(t)<x(t)\mbox{  and  }y^{\prime }(t)>0,\forall
t\in (0,1].
\end{equation}
\end{lemma}

\begin{proof}
Let us write $\tilde g(x,y)=(g_1(x,y),g_2(x,y))$ and let $G_2(x,y):=g_2(x,y)-y$. We may express $G_2$ as a power series in $x-\bar x$ and $y$ near $\bar p$ which converges in $B_\epsilon:=\{(x,y)\in \mathbb{R}^2:(x-\bar x)^2+y^2 \leq \epsilon^2\}$ with $\epsilon>0$ small.

If $G_2$ vanishes identically then $g_2(x,y)=y$ near $\bar p$. By preservation of area and the fact that $g_1(x,0)=(x,0),\forall x,$ we have $g_1(x,y)=x + yR(y)$ for a real-analytic function $R$ defined near $y=0$. Since ${\rm Fix }(\tilde g)=\mathbb{R}\times \{0\}$, $R$ does not vanish identically. The existence of $p_n$ as in the hypothesis implies $R(y)<0$ for $y$ small. In this case we can define the curve $\gamma_0$ by $\gamma_0(t)=(\bar x,t)$, with $t\geq 0$ small.

Now assume that $G_2$ does not vanish identically. We investigate the zeros of $G_2$ near $\bar p$ in $B_\epsilon$ for $\epsilon>0$ small. Notice that $\bar p \in \mathbb{R}\times \{0\} \cap B_\epsilon \subset \{G_2=0\}$ and thus $\bar p$ is not an isolated point of $\{G_2=0\}.$ Since $G_2$ is real-analytic we take $\epsilon>0$ small and find an even number of real-analytic embedded curves $\eta_i:[0,1] \to B_\epsilon$, $i=1\ldots 2m,$ with $\eta_i(0)=\bar p$, so that $\{G_2=0\} \cap B_\epsilon = \cup_{i=1}^{2m} {\rm Image}(\eta_i),$ see lemmas 3.1 and 3.3 of \cite{milnor}. Taking $\epsilon>0$ even smaller, we may assume that the image of any two of these curves intersect each other only at $\bar p$. Also, we may choose $\eta_1(t)=(\bar x + \epsilon t,0)$ and $\eta_2(t)=(\bar x - \epsilon t,0),t\in[0,1],$ since $\mathbb{R}\times \{0\}\subset \{G_2=0\}$. The existence of the sequence $p_n \to \bar p$ as in the hypothesis implies that $m\geq 2$ and therefore we find $j_0 \in \{3,\ldots,2m\}$ and a subsequence of $p_n$, still denoted by $p_n$, such that $p_n\in {\rm Image}(\eta_{j_0})$. Moreover, since $\eta_{j_0}(t)=(x_{j_0}(t),y_{j_0}(t))$ is real analytic, we have $y_{j_0}'(t)>0,\forall t\in (0,\mu],$ for some $\mu>0$ small, and, therefore, ${\rm Image}(\eta_{j_0}|_{[0,\mu]})$ projects injectively into the $y$-axis. Finally, we define $\gamma_0(t)=\eta_{i_0}(\mu t),t\in[0,1]$. By the properties of $p_n$ and the fact that ${\rm Fix }(\tilde g)=\mathbb{R}\times \{0\}$, we get that $\gamma_0$ satisfies the desired properties as in the statement. \end{proof}

To prove Theorem \ref{main1} we argue indirectly. Assume that there exists a
sequence $\epsilon _n\to 0^{+}$ so that $\tilde f_{\epsilon _n}$, defined as
in \eqref{translation}, admits a fixed point $p_n$. By the periodicity of $p_1 \circ \widetilde{f}(x,y)-x$ in $x$ we can assume that $p_n\to \bar p=(\bar x,0)\in
\mathbb{R}\times \{0\}$. This implies that $\tilde f$, restricted to a
neighbourhood $\tilde U$ of $\mathbb{R}\times \{0\}$, satisfies the conditions
of Lemma \ref{lemfix}. So we find a real-analytic curve $\gamma _0:[0,1]\to
\tilde A$, $\gamma _0(t)=(x(t),y(t))$, so that $\tilde f\circ \gamma
_0(t)=(w(t),y(t))$ satisfies \eqref{lemres}. In what follows, the curve $\gamma_0$ will be prolonged to a
Brouwer line $\tilde L$ in $\tilde A$ satisfying one of the possibilities:
\begin{itemize}
\item $\tilde L$ hits $\R \times \{1\}$. Since $\tilde f$ moves points in $\R \times \{1\}$ to the right and moves $\gamma_0$ to the left, $\tilde L$ must intersect its image, a contradiction.
\item $\tilde L$ is eventually periodic. In this case we obtain an area-preserving diffeomorphism of the closed annulus with a homotopically non-trivial simple closed curve which is disjoint from its image, again a contradiction.
\item $\tilde L$ is bounded and accumulates at $\R \times \{0\}$. In this case we obtain an area-preserving homeomorphism of the $2$-sphere admitting a simple closed curve bounding a topological disk whose image is properly contained inside itself, a contradiction.
\end{itemize}

Given $t\in (0,1]$ let $B_t:=\gamma _0(t),C_t:=\tilde f(B_t)$ and $%
A_t:=\tilde f^{-1}(B_t)$. Let $B_tC_t$ be the horizontal segment connecting $%
B_t$ and $C_t$, and let $A_tB_t:=\tilde f^{-1}(B_tC_t)$ be its inverse image. We
claim that $A_tB_t$ is a translation arc for $\tilde f$ if $t$ is small
enough. To see this, assume this is not the case so that we can find a point
$B_t\neq z\in B_tC_t$ which is also contained in $A_tB_t$. It follows that $%
p_1(z)>p_1(B_t)$. Since $\frac{ \partial (p_1 \circ\tilde f)}{\partial x}(x,y)\to 1$
as $(x,y)\to \bar p$, we have $p_1 \circ \tilde f(z)-p_1 \circ \tilde f(B_t)=\frac{\partial
(p_1 \circ \tilde f)}{\partial x}(\xi )(p_1(z)-p_1(B_t))>0$, for some $\xi \in B_tC_t$
and $t>0$ small. Since $\tilde f(z)\in B_tC_t$, we also have $p_1 \circ \tilde
f(z)\leq p_1(C_t)=p_1 \circ \tilde f(B_t)$. This leads to a contradiction which
proves that indeed $A_tB_t$ is a translation arc for $\tilde f$ and
\begin{equation}
\label{desigp1}p_1(z)>p_1(B_t),
\end{equation}
for all internal points $z\in A_tB_t$ where $t\in (0,1]$ is fixed and small.

Let us fix a sufficiently small $t_0>0$ such that for some number $c_0 \in
\mathbb{R},$ $\widetilde{f}^{-1}(\gamma _0([0,t_0]))\cup \gamma
_0([0,t_0])\cup \widetilde{f}(\gamma _0([0,t_0]))\cup B_{t_0}C_{t_0}\cup
A_{t_0}B_{t_0}$ is disjoint from all the verticals $V_{k+c_0}=\{(x,y)\in
\widetilde{A}:x=k+c_0\},k\in \mathbb{Z}.$ We may assume without loss of
generality that $c_0=0.$

In order to directly apply results from Ker\'ekj\'art\'o's construction of Brou\-wer lines in the plane as stated in Proposition~\ref{propresume},
we consider the homeomorphism $d:{\rm interior}(\tilde A)\to
\mathbb{R}^2$ given by $d(x,y)=(x,\frac{y-1/2}{y(1-y)})$ and the induced
orientation preserving homeomorphism $h:\mathbb{R}^2\to \mathbb{R}^2$ given
by $h=d\circ \tilde f\circ d^{-1}$. From the hypothesis ${\rm Fix}(\tilde f)= \R \times \{0\}$, we get ${\rm Fix}(h)= \emptyset$.

Let $A:=d(A_{t_0}),B:=d(B_{t_0})$ and $C:=d(C_{t_0}).$ Denote by $AB$ the
simple arc $d(A_{t_0}B_{t_0})$ and by $BC$ its image under $h$. Notice that $%
AB$ is a translation arc and that $BC$ is a horizontal simple arc. From
\eqref{desigp1}, the vertical line through $B$ intersects $AB$ and $BC$ only
at $B$. Hence we can start the construction of a Brouwer line for $h$ with the
vertical line starting from $B$ towards the upside and proceeding as in
Section \ref{SecKerek}, thus obtaining the half line $L_1$. To obtain $L_2$
we simply define it by $L_2=d(\gamma _0|_{(0,t_0]})$. It follows that $%
L=L_1\cup L_2$ is a Brouwer line, see Remark \ref{remL2}. Let $\tilde{L}
:=d^{-1}(L)=\tilde L_1 \cup \tilde L_2$ and observe that
\begin{equation}
\label{BL}\tilde f(\tilde{L})\cap \tilde{L}=\emptyset
\end{equation}

\begin{figure}\label{figu3}
\begin{center}
\includegraphics[width=65mm]{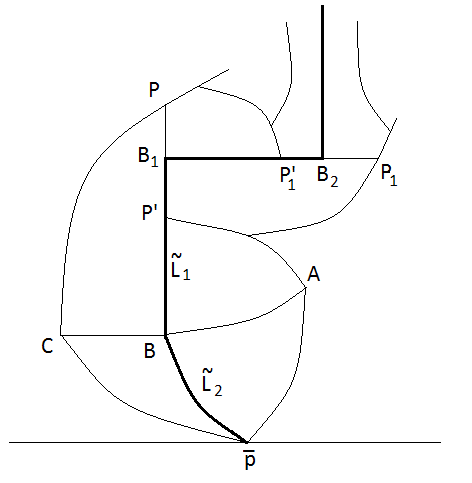}
\end{center}
\caption{The Brower line $\tilde L= \tilde L_1 \cup \tilde L_2\subset \tilde A.$}
\end{figure}

Now we prove that the existence of the Brouwer line $\tilde{L}$ leads to
a contradiction. First, from the twist
condition \eqref{twist}, we can find $0<\delta <1$ such that for all $%
(x,y)\in S_\delta :=\{(x,y)\in \tilde A:\delta \leq y\leq 1\}$, we have $%
p_1\circ \tilde f(x,y)>x$ and $p_1\circ \tilde f^{-1}(x,y)<x$. This implies
that $h$ satisfies condition \eqref{twistinfinito} for $y_0=\frac{\delta
-1/2 }{\delta (1-\delta )}$, see Remark \ref{remboun}. It follows that if $
\tilde{L}$ hits $S_\delta $ then $\overline{\tilde{L}}$ contains a
vertical segment with end point  $z_0\in \mathbb{R}\times \{1\}$. By
construction, points of $\overline{\tilde{L}}$ near but different from $%
\bar p$ are mapped under $\tilde f$ to the left, while points near $z_0$ are
mapped to the right. This implies that $\tilde f(\tilde{L})\cap
\tilde{L}\neq \emptyset $, which contradicts \eqref{BL}. Hence we can
assume that $\tilde{L}$ does not accumulate at $\mathbb{R}\times \{1\}$.

Since $\tilde{f}$ is analytic and $p_1\circ
\tilde{f}(x,1)>x,\forall x\in \mathbb{R},$ we get that $\tilde{f}^{-1}(V_k)\cap V_k=(
\tilde{f}^{-1}(V_0)\cap V_0)+(k,0)$ is a finite set for all $k \in \Z$ and, therefore, condition \eqref{finitude} holds for $h$. This implies that if $\widetilde{L}$ is horizontally unbounded then, as explained in Remark \ref{remdev}, we can find $N\in \mathbb{Z}$ and another Brouwer line $\tilde
L_{{\rm per}}=d^{-1}(L_{{\rm per}})\subset {\rm interior}(\tilde A)$ which
is $N$-periodic in $x$, i.e., $\tilde L_{{\rm per}}+(kN,0)=\tilde L_{{\rm per}%
},\forall k\in \mathbb{Z}$. Let $f_N:A_N\to A_N$ be the map induced by $%
\tilde f$ on the annulus $A_N:=\tilde A/T_N$, where $T_N:\tilde A\to \tilde A
$ is the horizontal translation $T_N(x,y)=(x+N,y)$. Let $p_N:\tilde A\to A_N$
be the associated covering map and let $L_N:=p_N(\tilde L_{{\rm per}})$. From the
properties of $\tilde L_{{\rm per}}$ and of the map $\tilde f$ we see that $%
L_N$ and $f_N(L_N)$ are disjoint simple closed curves which are
homotopically non-trivial. Let $C_{-}$ be the topological closed annulus
bounded by $L_N$ and $p_N(\mathbb{R}\times \{0\})$. Then either $f_N$ or $%
f_N^{-1}$ maps $C_{-}$ properly into itself. Since $f_N$ preserves a finite
area form, we get a contradiction. Hence we can assume that $%
\tilde L$ is horizontally bounded and accumulates only at $\mathbb{R}\times
\{0\}$.

Now we find $N_0\in \mathbb{N}$ large enough so that $\tilde L\cap
(\tilde L+(N_0,0))=\emptyset ,$ which implies by Brouwer's lemma (see for
instance \cite{brown}) that \begin{equation}\label{brouwerlema} \tilde
L\cap (\tilde L+(kN_0,0))=\emptyset, \forall k\in \mathbb{Z}^*.\end{equation} As before we consider the
annulus $A_{N_0}:=\tilde A/T_{N_0}$ and identify the points in each
component of $\partial A_{N_0}$ to obtain a topological sphere $S^2$. We end
up with a map $\widehat{f}_{N_0}:S^2\to S^2$ induced by $f_{N_0}$ which
preserves orientation and a finite area form. The closure of $p_{N_0}(\tilde L)$
corresponds to a simple closed curve $\gamma _0\subset S^2$ passing through
the pole $p_0$, which corresponds to the lower component of $\partial A_{N_0}
$. This last assertion follows from \eqref{brouwerlema}. Since $\tilde L$ is a Brouwer line we see that $\widehat{f}_{N_0}(\gamma
_0)\cap \gamma _0=\{p_0\}$ and that $\widehat{f}_{N_0}$ maps properly one
component of $S^2\setminus \gamma _0$ into itself. This contradicts the preservation of a finite area form and shows that $\tilde L$ cannot exist.
The proof of Theorem \ref{main1} is complete.

\section{Proof of Theorem \ref{main2}}

\label{secrs}

Our aim in this section is to construct an annulus diffeomorphism $f:S^1
\times [0,1] \to S^1 \times [0,1]$, homotopic to the identity map, which
satisfies

\begin{itemize}
\item[(i)]  $f$ is smooth and area-preserving.

\item[(ii)]  $\text{Fix}(f)=S^1\times \{0\}$.

\item[(iii)]  If $\tilde f:\mathbb{R}\times [0,1]\to \mathbb{R}\times [0,1]$
is the lift of $f$ satisfying $\tilde f(x,0)=(x,0),\forall x\in \mathbb{R}$,
then $p_1\circ \tilde f(x,1)>x,\forall x\in \mathbb{R}.$

\item[(iv)]  Given $\epsilon >0$, if $f_\epsilon :S^1\times [0,1]\to
S^1\times [0,1]$ is the map induced by the lift $\tilde f_\epsilon :=\tilde
f+(\epsilon ,0)$ as before, then there exists a positive sequence $(\epsilon _n)_{n\in \mathbb{N}}$ with
$\epsilon _n\to 0^{+}$ as $n\to \infty $, such that $\text{Fix}(\tilde f_{\epsilon
_n})\neq \emptyset ,\forall n$.
\end{itemize}

As proved in Theorem \ref{main1}, such a diffeomorphism cannot exist if
smoothness is replaced by real-analyticity in (i).

\subsection{Area-preserving maps and generating functions}

We start by recalling basic facts on area preserving maps associated to
generating functions. Let $U:= \{(X,y) \in \mathbb{R}\times [0,1]: X^2 + y^2
<\varepsilon \}$, $\varepsilon>0$, and $g:U \to \mathbb{R}$ be a smooth
function so that
\begin{equation}
\label{hipg} D^{\nu} g|_{\{y=0\}\cap U} \equiv 0,\forall 0\leq |\nu| \leq 2.
\end{equation}
We denote by $G:U \to \mathbb{R}$ the function given by
\begin{equation}
\label{generating} G(X,y):=Xy - g(X,y).
\end{equation}
Let $(x,Y)\in \mathbb{R}\times [0,1]$ be given by
\begin{equation}
\label{xY}
\begin{aligned} x:=  G_y = & X - g_y(X,y), \\ Y:=  G_X = & y - g_X(X,y).\end{aligned}
\end{equation}

We see from the first equation of \eqref{xY} and the hypothesis \eqref{hipg}
on $g$ that we can use the implicit function theorem to write $X=\alpha(x,y)$
for $|(x,y)|$ small, where $\alpha$ is a smooth map satisfying $%
\alpha(x,0)=x, \forall x$. In this case $Y=y+g_X(\alpha(x,y),y)=\beta(x,y)%
\geq 0$, where $\beta$ is smooth and satisfies $\beta(x,0)=0,\forall x$. Let
us denote by $\bar f:V\subset \mathbb{R}\times [0,1] \to \mathbb{R}\times
[0,1]$ the map given by
$$
(X,Y)=\bar f(x,y):= (\alpha(x,y),\beta(x,y)),%
$$
where $V$ is a small neighborhood of $(0,0)\in \mathbb{R}\times [0,1]$. We
say that $\bar f$ is a local map associated to the generating function $G$.
Moreover, $\bar f|_{\mathbb{R}\times \{0\}}$ is the identity map.

\begin{prop}
The map $\bar f$ preserves the area form $dx\wedge dy$ on $\mathbb{R}\times
[0,1]$, i.e.,
$$
dX\wedge dY=dx\wedge dy.
$$
\end{prop}

\begin{proof}From \eqref{xY} we have \begin{equation} \label{extd} \begin{aligned}dx = &  (1+g_{yX})dX + g_{yy}dy & \Rightarrow & \hspace{0.2cm} dx \wedge dy= (1+g_{yX})dX \wedge dy, \\ dY = & (1+g_{Xy})dy + g_{XX}dX & \Rightarrow & \hspace{0.2cm} dX \wedge dY =  (1+g_{Xy})dX \wedge dy.\end{aligned}\end{equation}Since $g$ is smooth, the proposition follows. \end{proof}

\subsection{A special generating function}

Let $\rho:[0,\infty) \to [0,1]$ be a smooth function satisfying $\rho\equiv
1 $ in $[0,\frac{1}{16}]$, $\rho\equiv 0$ in $[\frac{1}{4},\infty)$ and $%
\rho^{\prime}<0$ in $(\frac{1}{16},\frac{1}{4})$. We define the vector field
$X$ on the strip $W_1 := \mathbb{R}\times [-1,1]$ by
$$
X(x,y) = \rho(x^2+y^2)\cdot (-y,x).%
$$
It is clear that $X$ is smooth, $X\equiv 0$ outside $B(1/2):=\{(x,y)\in
W_1:x^2+y^2 \leq 1/4\}$ and $X(x,y)=(-y,x)$ in $B(1/4)$.

The flow $\{\varphi_t\}$ of $X$ on $W_1$ is defined for all $t\in \mathbb{R}$
and satisfies
\begin{equation}
\label{fluxo}%
\begin{aligned} & \varphi_t(x,y)=  (x,y),\forall (x,y)\in W_1 \setminus B(1/2),\forall t. \\ & \varphi_{\pi}(x,y)=  -(x,y),\forall (x,y)\in B(1/4).\end{aligned}
\end{equation}

Now let $W_0: = \mathbb{R}\times [0,1]$. For each $k \in \mathbb{N},$ let $%
F_k:=\mathbb{R}\times (\frac{1}{2^{k+1}},\frac{1}{2^k}]\subset W_0,
F_{\infty}:=\mathbb{R}\times \{0\}\subset W_0$ and consider the
diffeomorphism
\begin{equation}
\label{tk}
\begin{aligned}  t_k:  F_k & \to W_1\setminus \mathbb{R}\times \{-1\},k \in \mathbb{N},\\  (x,y) & \mapsto (2^{k+2}x,2^{k+2}y-3).\end{aligned}
\end{equation}
Letting $\partial_k^+ := \mathbb{R}\times \{1/2^{k}\}\subset F_k$, we
observe that $t_k(\partial_k^+) = \mathbb{R}\times \{1\}$.

Next let us define a map $\psi:W_0 \to W_0$ by
\begin{equation}
\label{eqpsi}
\begin{aligned}  \psi|_{F_k}:= & t_k^{-1} \circ \varphi_{\pi}  \circ t_k, \forall k, \\ \psi|_{F_{\infty}}:= & \text{Id}|_{F_{\infty}}.  \end{aligned}
\end{equation}

Let $p_k:=(0,\frac{3}{2^{k+2}})\in F_k\subset W_0$ be the `midpoint' of $F_k$%
. From \eqref{fluxo}, \eqref{tk} and \eqref{eqpsi} we note that
\begin{equation}
\label{suporte} \text{supp} \psi = \overline{\bigcup_{k\geq 0}
B_{p_k}(1/2^{k+3})},
\end{equation}
where $B_p(r)$ denotes the closed ball centered at $p$ with radius $r$.
Moreover, $\psi$ is the identity map when restricted to a small
neighbourhood of each $\partial_k^+, \forall k.$ These observations together
with the second equation in \eqref{eqpsi} imply that $\psi$ is smooth in $%
W_0 \setminus \{(0,0)\}.$ We also get that $\psi$ is a diffeomorphism when
restricted to $W_0 \setminus \{(0,0)\}.$ From the second equation in
\eqref{fluxo} we have
\begin{equation}
\label{inverte} \psi(x,y)= 2p_k-(x,y), \forall (x,y)\in
B_{p_k}(1/2^{k+4}),\forall k.
\end{equation}

Let $h:[0,1]\to [0,\infty)$ be the smooth function given by
\begin{equation}
\label{defih}
\begin{aligned} h(t)= & e^{-1/t},\forall t>0,\\ h(0) = & 0.\end{aligned}
\end{equation}
Note that $h$ is flat at $t=0$, i.e.,
\begin{equation}
\label{flat} h^{(n)}(0)=0, \forall n.
\end{equation}
Observe also that given $l \in \mathbb{N}$, we find polynomial functions $%
P_l,Q_l$ such that
\begin{equation}
\label{derivah} h^{(l)}(t)= e^{-1/t} \frac{P_l(t)}{Q_l(t)},\forall t>0.
\end{equation}
This can easily be proved by induction.

Now let $g:W_0 \to [0,\infty)$ be defined by
\begin{equation}
g:=h \circ p_2 \circ \psi.
\end{equation}

\begin{prop}
\label{propsmooth}We have the following:

\begin{itemize}
\item[(i)]  The function $g$ is smooth and $D^\nu g|_{\mathbb{R}\times
\{0\}}\equiv 0, \forall \nu \geq 0$.

\item[(ii)]  The set $\text{Crit}(g)$ of critical points of $g$ coincides
with $\mathbb{R}\times \{0\}$.

\item[(iii)]  There exists a positive sequence $(s_k)_{k \in \mathbb{N}}$
satisfying $s_k \to 0^+$ as $k \to \infty$ such that $\nabla
g|_{\partial_k^+}=(0,s_k),\forall k.$

\item[(iv)]  There exists a positive sequence $(m_k)_{k \in \mathbb{N}}$
satisfying $m_k \to 0^+$ as $k \to \infty$ such that $\nabla g
(p_k)=(0,-m_k),\forall k.$
\end{itemize}
\end{prop}

Postponing its proof to Section \ref{proof} below, we use Proposition \ref
{propsmooth} to show that $g$ induces a diffeomorphism $f:S^1 \times [0,1]
\to S^1 \times [0,1]$ satisfying conditions (i)-(iv) as described in the beginning of Section
\ref{secrs}.

Let $G(X,y) = Xy + g(X,y)$ be the function defined in \eqref{generating}.
Then, as explained before, we find a small neighbourhood $V$ of $(0,0)$ in $
\mathbb{R}\times [0,1]$ and a smooth area-preserving map $\bar f:V\to
\mathbb{R}\times [0,1], \bar f(x,y)=(X,Y),$ so that $\bar f|_{V \cap F_{\infty}}$ is the
identity map, i.e., $\bar f$ is the local map associated to the generating
function $G$. From \eqref{xY} we see that the fixed points of $\bar f$
correspond to critical points of $g$. This implies that $\text{Fix} (\bar f)
= \text{Crit}(g)= V \cap F_{\infty}$. Now observe that since $\nabla
g|_{\partial_k^+}=(0,s_k)$, with $s_k \to 0^+$ as $k \to \infty$, we have
from \eqref{xY} that $\bar f(x,y)=(x+s_k,y), \forall (x,y)\in \partial_k^+$.
In the same way, since $\nabla g(p_k)=(0,-m_k),$   we have $\bar f(m_k,3/2^{k+2}) = p_k= (0,3/2^{k+2}),\forall k,$ with $m_k
\to 0^+$ as $k \to \infty$. This implies that for all $k\geq 0$, the map $\bar f + (m_k,0)$ admits $(m_k,3/2^{k+2})$ as a fixed point. From
\eqref{suporte} and the definition of $g$, we see that given any $x_1>0$
small we have $\bar f(x,y)=(x+h^{\prime}(y),y),\forall (x,y) \in
\{|x|=x_1,0\leq y \leq 2x_1\}.$ Given $\lambda>0$, let $T_{\lambda}:
\mathbb{R}\times [0,\infty) \to \mathbb{R}\times [0,\infty)$ be the map $
T_{\lambda}=(\lambda x, \lambda y).$ If necessary we replace $\bar f$ by $
(T_{1/2^{k_0}})^{-1} \circ \bar f \circ T_{1/2^{k_0}}$, for a fixed $k_0$
sufficiently large, in order to find a map defined in $[-1/2,1/2] \times
[0,1]$ with the same properties above. Identifying $(-1/2,y) \in \{-1/2\}
\times [0,1]$ with $(1/2,y) \in \{1/2\} \times [0,1]$ we finally find an
annulus map $f:S^1 \times [0,1] \to S^1 \times [0,1]$ with all the desired
properties.

Notice that the diffeomorphism induced by the generating function $g$ is defined in the open neighbourhood $V\subset \R \times [0,1]$ which might be very small. This explains why property (iii) is necessary in Proposition \ref{propsmooth}. Its proof is not straightforward and is left for the next section.

\subsection{Proof of Proposition \ref{propsmooth}}

\label{proof}

As observed before, $\psi$ is smooth on $W_0 \setminus \{(0,0)\}$. Hence $g$
is smooth in this set as well. Moreover, since $\psi$ is the identity map
near $(\bar x,0),$ for each $\bar x \neq 0$, we have that $g$ is given by $%
g(x,y)=h(y)$ near $(\bar x, 0)$. It follows from \eqref{flat} that
\begin{equation}
\label{dx1} D^\nu g(\bar x,0) = 0,\forall \bar x \neq 0, \forall |\nu|\geq
0.
\end{equation}

It remains to prove that $g$ is smooth at $(0,0)$ and that $D^\nu g(0,0)=0,
\forall |\nu| \geq 0$. Let $p=p_2 \circ \psi$. From the definition of $\psi$
we see that
\begin{equation}
\label{pescala} p(x,y)=\frac{1}{2^n}p(2^nx,2^ny),\forall(x,y)\in
W_0\setminus F_{\infty}.
\end{equation}

For any given smooth function $a:U\subset \mathbb{R}^2 \to \mathbb{R}$, we
denote by $D^{\alpha}a=\frac{\partial ^{|\alpha|}a}{\partial x^i \partial y^j%
}$ where $\alpha=(i,j) \in \mathbb{N}^2$ and $|\alpha|=i+j.$

\begin{lemma}
\label{lemmap1} In $W_0\setminus F_\infty $, we have
\begin{equation}
\label{derivag}D^\alpha g=\sum_{l=1}^{|\alpha |}h^{(l)}(p)T_{\alpha
,l}(p_x,p_y,\ldots ,D^\beta p),
\end{equation}
where $T_{\alpha ,l}$ is a multi-variable polynomial function on $D^\beta p$
with $\beta $ satisfying $1\leq |\beta |\leq |\alpha |-l+1$.
\end{lemma}

\begin{proof}Observe that $D^{(1,0)}g=h'(p) p_x$ and $D^{(0,1)}g=h'(p) p_y$ which have the form above with $$\begin{aligned} T_{(1,0),1}(p_x,p_y)=p_x\\T_{(0,1),1}(p_x,p_y)=p_y.\end{aligned}$$ In the same fashion we have $D^{(2,0)}g=h'(p)p_{xx}+h''(p)p_x^2,$ $D^{(1,1)}g=h'(p)p_{xy}+h''(p) p_xp_y,$ $D^{(0,2)}g=h'(p)p_{yy}+h''(p) p_y^2$ and $$\begin{aligned} T_{(2,0),1}(p_x,p_y,p_{xx},p_{xy},p_{yy})= & p_{xx}, \hspace{0.5cm} T_{(2,0),2}(p_x,p_y)=  p_x^2,\\T_{(1,1),1}(p_x,p_y,p_{xx},p_{xy},p_{yy})= & p_{xy}, \hspace{0.5cm} T_{(1,1),2}(p_x,p_y)=  p_xp_y,\end{aligned} $$ and so on.
Let $\tilde \alpha = \alpha + (1,0)$ and observe that $D^{\tilde \alpha} g = D^{(1,0)}D^\alpha g$. The case $\tilde \alpha = \alpha + (0,1)$ is similar. Now an easy induction argument establishes the claim. \end{proof}

It follows from \eqref{derivah} and \eqref{derivag} that
\begin{equation}
\label{Dalfag} D^\alpha g = e^{-1/p}\sum_{l=1}^{|\alpha|} \frac{P_l(p)}{%
Q_l(p)}T_{\alpha,l}(p_x,p_y,\ldots,D^\beta p),
\end{equation}
where $P_l,Q_l$ are polynomial functions in $p$.

\begin{lemma}
\label{lemmap2}There are constants $C_\beta >0$ depending on $(0,0)\neq
\beta \in \mathbb{N}\times \mathbb{N}$ such that
$$
|D^\beta p(x,y)|\leq \frac{C_\beta }{y^{|\beta |}},\forall (x,y)\in
W_0\setminus F_\infty .
$$
\end{lemma}

\begin{proof}  Given $(x,y)\in W_0 \setminus F_{\infty}$, let $n(x,y)\in \mathbb{N}$ be the unique positive integer such that $2^{n(x,y)}(x,y)\in F_0=\mathbb{R}\times (1/2,1]$. From \eqref{pescala}, we have $$D^\beta p(x,y)=2^{n(x,y)|\beta|-1}D^\beta p(2^{n(x,y)}x,2^{n(x,y)}y).$$ Let $$0<C_\beta: =
\sup_{(x,y)\in F_0} D^\beta p(x,y)< \infty.$$ It follows from the definition of $n(x,y)$ that $$\begin{aligned} |D^\beta p(x,y)| \leq & 2^{n(x,y)|\beta|-1} C_\beta. \end{aligned}$$ Now since $2^{n(x,y)} \leq \frac{1}{y}\Rightarrow 2^{n(x,y)|\beta|-1}\leq \frac{1}{y^{|\beta|}}$, the claim follows. \end{proof}

\begin{lemma}
\label{lemmap3}$|D^\beta g(x,y)|\to 0$ as $(x,y)\to (0,0)$, $\forall \beta $.
\end{lemma}

\begin{proof}From \eqref{dx1} it suffices to consider $(x,y)\in W_0\setminus F_{\infty}$. From Lemmas \ref{lemmap1} and \ref{lemmap2} we find constants $C_{\alpha,l},n_{\alpha,l}>0$ so that \begin{equation} \label{talfa} |T_{\alpha,l}(p_x,p_y,\ldots,D^\beta p)|\leq \frac{C_{\alpha,l}}{y^{n_{\alpha,l}}}.\end{equation}
We  can also find constants $K_l,m_l>0$ so that \begin{equation}\label{tld}\left|\frac{P_l(p)}{Q_l(p)}\right| \leq \frac{K_l}{p^{m_l}}.\end{equation} Now since $0<y/2 \leq p(x,y) \leq 2y$, we get from \eqref{Dalfag}, \eqref{talfa} and \eqref{tld}, that \begin{equation}\label{desig} \begin{aligned} |D^\beta g(x,y)| \leq & e^{-\frac{1}{2y}} \sum_{l=1}^{|\beta|}\frac{2^{m_l}K_lC_{\alpha,l}}{y^{m_l+n_{\alpha,l}}} \\ \leq & e^{-\frac{1}{2y}}
\frac{M_\beta}{y^{m_\beta}}\to 0 , \end{aligned}\end{equation}  as  $y\to 0,$ where $M_\beta, m_\beta>0$ are suitable constants.\end{proof}

Let $\beta=(b_1,0),$ where $b_1 \in \mathbb{N}$. Since $g|_{F_{\infty}%
\setminus \{(0,0)\}} =0$, we have $D^\beta g(0,0)=0.$ From \eqref{desig}, $%
D^\beta g$ is continuous at $(0,0)$.

Now assume $b_2 > 0$ and let $\beta=(b_1,b_2)$. Then
$$
D^\beta g(0,0) = \lim_{y \to 0^+} \frac{D^{\beta-(0,1)} g(0,y) -
D^{\beta-(0,1)}g(0,0)}{y}.%
$$
Using induction on $b_2$ and inequality \eqref{desig} again, we find
\begin{equation}
\label{dzero} D^\beta g(0,0)=0,\forall \beta.
\end{equation}

Finally from \eqref{desig} and \eqref{dzero} we have that $D^\beta g$ is
continuous at $(0,0)$. The proof of (i) is finished.

It is clear from the considerations above that $\text{Crit}(g) \supseteq
F_{\infty}.$ Since $\psi$ is a local diffeomorphism in $W_0 \setminus
F_{\infty}$, $p_2$ is a submersion and $h^{\prime}(y)>0,\forall y>0,$ we get
that also $g$ is a submersion when restricted to $W_0 \setminus F_{\infty}$. This implies
that $\text{Crit}(g) \subseteq F_{\infty}$ and, therefore, $\text{Crit}%
(g)=F_{\infty}=\mathbb{R}\times \{0\}$. This proves (ii).

Since $\psi$ is the identity map near $\partial_k^+, \forall k,$ we have
that $g(x,y)=h(y)$ for all $(x,y)$ near $\partial_k^+$. This implies that
$$
\nabla g(x,y)=(0,h^{\prime}(1/2^k)), \forall (x,y)\in \partial_k^+.%
$$
Since $h^{\prime}(1/2^k)>0, \forall k$ and $\displaystyle \lim_{k \to
\infty} h^{\prime}(1/2^k) = 0$, (iii) follows.

To prove (iv), observe from \eqref{inverte} that $g(x,y)=h(p_2(2p_k-(x,y)))=
h(3/2^{k+1}-y)$ for all $(x,y)\in B_{p_k}(1/2^{k+4})$. This implies in
particular that
$$
\nabla g(p_k)=(0,-h^{\prime}(3/2^{k+2})).%
$$
Since $h^{\prime}(3/2^{k+2})>0,\forall k$ and $\displaystyle \lim_{k \to
\infty} h^{\prime}(3/2^{k+2}) = 0$, (iv) follows. The proof of Proposition
\ref{propsmooth} is now complete.

\vspace{0.5cm}

\noindent
\subsection*{Acknowledgements}  S. Addas-Zanata and P. Salom\~ao are partially supported by FAPESP grant 2011/16265-8. AZ and PS are partially supported by CNPq grants 303127/2012-0 and 303651/2010-5, respectively.

\end{document}